\documentclass[onefignum,onetabnum]{siamart251216}

\usepackage{bm}
\DeclareMathOperator*{\argmin}{arg\,min}
\DeclareMathOperator*{\argmax}{arg\,max}

\usepackage{algorithm}
\usepackage{algpseudocode}
\usepackage{algorithmicx}

\usepackage{lipsum}
\usepackage{amsfonts}
\usepackage[normalem]{ulem}
\usepackage{graphicx}
\usepackage{epstopdf}
\usepackage{pifont}
\usepackage{siunitx}
\usepackage[caption=false]{subfig}
\usepackage{pgfplotstable}
    \usetikzlibrary{
        pgfplots.colorbrewer,
    }
    \pgfplotsset{
        cycle list/.define={my marks}{
            every mark/.append style={solid,fill=\pgfkeysvalueof{/pgfplots/mark list fill}},mark=*\\
            every mark/.append style={solid,fill=\pgfkeysvalueof{/pgfplots/mark list fill}},mark=square*\\
            every mark/.append style={solid,fill=\pgfkeysvalueof{/pgfplots/mark list fill}},mark=triangle*\\
            every mark/.append style={solid,fill=\pgfkeysvalueof{/pgfplots/mark list fill}},mark=diamond*\\
        },
    }
\usepackage{algpseudocode}
\usepackage{multirow}
\usepackage{adjustbox}
\usepackage{booktabs,colortbl}
\usepackage{cprotect}
\usepackage{stmaryrd}
\usepackage{comment}

\ifpdf
  \DeclareGraphicsExtensions{.eps,.pdf,.png,.jpg}
\else
  \DeclareGraphicsExtensions{.eps}
\fi

\newsiamremark{remark}{Remark}
\newsiamremark{hypothesis}{Hypothesis}
\crefname{hypothesis}{Hypothesis}{Hypotheses}
\newsiamthm{claim}{Claim}

\headers{Shifted Linear Solver}{Authors}

\title{Shifted Linear Systems Solver\thanks{Submitted to the editors \today.}}
\author{
Hussam Al Daas\thanks{Scientific Computing Department, STFC, Rutherford Appleton Laboratory, Harwell Campus, Didcot, Oxfordshire, OX11 0QX, UK  (\email{hussam.al-daas@stfc.ac.uk}).}
and 
Davide Palitta\thanks{Dipartimento di Matematica, (AM)2 , Alma Mater Studiorum - Università di Bologna, Piazza di
Porta San Donato 5, 40126 Bologna, Italy(\email{davide.palitta@unibo.it}).}
}

\usepackage{amsopn}

\newcommand{\C}{\mathbb{C}}

\definecolor{darkpastelgreen}{rgb}{0.01, 0.75, 0.24}
\definecolor{burgundy}{rgb}{0.5, 0.0, 0.13}

\ifpdf
\hypersetup{
  pdftitle={Minimal residual rational Krylov subspace method for sequences of shifted linear systems},
  pdfauthor={Hussam Al Daas \and Davide Palitta}
}
\fi

\title{Minimal residual rational Krylov subspace method for sequences of shifted linear systems}
\author{
Hussam Al Daas\thanks{Scientific Computing Department, STFC, Rutherford Appleton Laboratory, Harwell Campus, Didcot, Oxfordshire, OX11 0QX, UK  (\email{hussam.al-daas@stfc.ac.uk}).}
\and 
Davide Palitta\thanks{Dipartimento di Matematica, (AM)2 , Alma Mater Studiorum - Università di Bologna, Piazza di
Porta San Donato 5, 40126 Bologna, Italy (\email{davide.palitta@unibo.it}).}
}
\newtheorem{experiment}[theorem]{Example}
\let\oldexperiment\experiment
\renewcommand{\experiment}{\oldexperiment\normalfont}

\begin{document}

\maketitle
\renewcommand{\thefootnote}{\fnsymbol{footnote}}
\maketitle \pagestyle{myheadings} \thispagestyle{plain}
\markboth{ H.\ AL DAAS, D.\ PALITTA}{MINIMAL RESIDUAL RATIONAL KRYLOV FOR SHIFTED LINEAR SYSTEMS}

% REQUIRED
\begin{abstract} The solution of sequences of shifted linear systems is a classic problem in numerical linear algebra, and a variety of efficient methods have been proposed over the years. Nevertheless, there exist challenging scenarios {\color{blue} for which there is still} a lack of {\color{blue}performant} solvers. For instance, state-of-the-art procedures {\color{blue}based on restarting may} struggle to handle nonsymmetric problems where the shifts are {\color{blue}possibly} complex numbers that do not come as conjugate pairs. We design a novel projection strategy based on the rational Krylov subspace equipped with a minimal residual condition. We also devise a novel pole selection procedure, tailored to our problem, providing poles for the rational Krylov basis construction that yield faster convergence than those computed by available general-purpose schemes.
A panel of diverse numerical experiments shows that our novel approach performs better than state-of-the-art techniques, especially on the very challenging problems mentioned above.
\end{abstract}

%%%%%%%%%%%%%%%%%%%%%%%%%%%%%%%%%%
% REQUIRED

% REQUIRED

%%%%%%%%%%%%%%%%%%%%%%%%%%%%%%%%%%
\section{Introduction}\label{sec:introduction}

The solution of sequences of shifted linear systems of the form
\begin{equation}\label{eq:main}
 (A+s_i I_n)x^{(i)} = b^{(i)}, \quad i=1,\ldots,\ell,
\end{equation}
    with $A\in{\color{blue}\mathbb{C}}^{n\times n}$, $I_n$ the identity matrix of order $n$, $s_i$ a possibly complex number, and $b^{(i)}\in\mathbb{R}^n$ is a classic problem in numerical linear algebra. For the sake of simplicity in the presentation, we assume the right-hand side in~\eqref{eq:main} to be independent of the parameter index $i$, namely $b\equiv b^{(i)}$ for all $i=1,\ldots,\ell$. However, many of the results presented in this {\color{blue} paper} can be extended to a more general framework; see section~\ref{sec:multiple_rhs}.

Equations~\eqref{eq:main} {\color{blue} is encountered} in numerous application settings, such as control theory and model order reduction techniques~\cite{MOR2015}, structural dynamics~\cite{FERIANI2000},
lattice quantum-chromodynamics~\cite{Frommer1}, Tikhonov regularization~\cite{Frommer2},
the solution of time-dependent PDEs~\cite{TimeDepPDES}, and the inner step in sophisticated eigensolvers~\cite{FEAST}, to name a few.

The ubiquity of~\eqref{eq:main} has driven the interest of our community into developing many different strategies for the solution of this problem over the years; see, e.g., the survey~\cite{Survey} and the references therein. Most of these techniques build upon the construction of the polynomial Krylov subspace
\begin{equation}\label{eq:PK}
\mathbf{K}_m(A,b)=\text{span}\{{\color{blue}b},A{\color{blue}b},\ldots,A^{m-1}{\color{blue}b}\}.
\end{equation}
%%
%where $r_0=b-Ax_0$ is the initial residual vector for a given initial guess $x_0\in\mathbb{R}^n$.
In these schemes, solutions of the form
\begin{equation}\label{eq:sol}
 x_m^{(i)}=V_my_m^{(i)},
\end{equation}
 are constructed where the columns of $V_m=[v_1,\ldots,v_m]\in\mathbb{C}^{n\times m}$ form an orthonormal basis of~\eqref{eq:PK} whereas $y_m^{(i)}\in\mathbb{C}^m$. The backbone of polynomial Krylov subspace methods for shifted linear systems is the
{\color{blue}\emph{shift-invariance}} property of~\eqref{eq:PK}, namely the following Arnoldi relation
\begin{equation}\label{eq:shiftArnoldi}
(A+s_iI_n)V_m=V_m(H_m+s_iI_m)+h_{m+1,m}v_{m+1}e_m^T=V_{m+1}\left(\underline{H}_m+s_i \begin{bmatrix}
I_m\\
0\\                                                                                     \end{bmatrix}\right),
\end{equation}
holds true for any shift $s_i$ in~\eqref{eq:main}. {\color{blue} The underlined matrix $\underline{X}$ in this work denotes a matrix formed by concatenating a square matrix $X$ with an additional last row or a block row for block methods.}
In particular, in~\eqref{eq:shiftArnoldi}, $H_m\in{\color{blue}\mathbb{C}}^{m\times m}$ and $h_{m+1,m}\in{\color{blue}\mathbb{C}}$ are quantities related to the orthonormalization step of the Arnoldi process and $\underline{H}_m=[H_m;h_{m+1,m}e_m^T]\in{\color{blue}\mathbb{C}}^{(m+1)\times m}$.
The relation~\eqref{eq:shiftArnoldi} explains the use of the shift-independent space~\eqref{eq:PK} employed in the construction of the solutions~\eqref{eq:sol}. Indeed,~\eqref{eq:shiftArnoldi} allows us to construct a single basis $V_m$ and deal with shifted quantities only at the projected level.

Once $V_m$ is formed, the second ingredient in the construction of the solution $x_m^{(i)}$ is the vector $y_m^{(i)}$ which contains the coefficients of the linear combination of the basis vectors providing the current solution. The computation of this vector $y_m^{(i)}$ is one of the main aspects characterizing a given Krylov subspace method. Some methods impose a Galerkin condition on the residual vectors $r_m^{(i)}=b-(A+s_iI)x_m^{(i)}$, namely they require the $r_m^{(i)}$'s to be orthogonal to $\mathbf{K}_m(A,b)$. This condition is equivalent to computing $y_m^{(i)}$ as the solution of the projected shifted linear system
$$(H_m+s_iI)y_m^{(i)}=\beta e_1,\quad \beta=\|{\color{blue}b}\|;$$
see, e.g.,~\cite{Saad2003}.

Other methods impose a minimal residual condition and compute $y_m^{(i)}$ as
$$y_m^{(i)}=\argmin_{y\in\mathbb{C}^m}\|{\color{blue}b}-(A+s_i I)V_my\|=\argmin_{y\in\mathbb{C}^m}\left\|\beta e_1 - \left(\underline{H}_m+s_i \begin{bmatrix}
I_m\\
0\\                                                                                     \end{bmatrix}\right)y\right\|,
$$
see, e.g.,~\cite{Saad2003}.

In many cases, polynomial Krylov subspace methods for~\eqref{eq:main} require the construction of a large subspace to meet the prescribed level of accuracy. This leads to severe drawbacks in terms of both computational cost and storage demand, especially for those methods where a full orthogonalization step and the allocation of the whole $V_m$ are necessary. Different strategies have been developed to overcome these issues. {\color{blue} Even though some contributions aiming to construct effective preconditioning operators for~\eqref{eq:main} can be found in the literature (see, e.g.,~\cite{Bau2015,Saibaba2013,Bakhos2017}), preconditioning is seldom an option as designing good preconditioners for all the shifted linear systems in the sequence~\eqref{eq:main} is rather difficult. Therefore, the restarting paradigm is often the preferred option in this setting; see, e.g.,~\cite{FOM,RestartedGMRES,BicgStab}.} However,
it is well-known that restarting needs collinearity of the residual vectors $r_m^{(i)}$ to be successful and this feature may be difficult to achieve in minimal residual methods unless rather strict conditions on $A$ and the shifts $s_i$ are considered; see, e.g.,~\cite{BicgStab}.

Even though they have been designed for solving more general problems like general parametrized linear systems, recycling techniques can also be employed for the solution of shifted linear systems. The main underlying idea here is to solve a given problem, e.g., by a Krylov subspace method, and reuse the most significant part of this approximation space in the solution of another problem. See, e.g.,~\cite{Rec1,rec2,Gaul2014,AldGHR21} for more details on recycling techniques. 

In principle, one could combine restart and recycling. However, this integration is not always possible due to the need for residual collinearity to restart; see, e.g.,~\cite[Section 7.1.1]{Survey} and~\cite[Theorem 1]{rec2}.

A completely different approach consists {\color{blue}of} employing approximation spaces different from~\eqref{eq:PK} in the construction of the solutions~\eqref{eq:sol}. This novel point of view has been proposed for the first time in~\cite{Simoncini2010} where tools coming from the matrix equation literature have been adapted to the solution of~\eqref{eq:main}. Indeed, it is easy to show that the sequence of shifted linear systems in~\eqref{eq:main} with $b\equiv b^{(i)}$, for all $i$, can be reformulated in terms of a single Sylvester matrix equation of the form
\begin{equation}\label{eq:Sylv}
 AX+XS=b\mathbf{1}_\ell^T,
\end{equation}
where $X=[x^{(1)},\ldots,x^{(\ell)}]\in\mathbb{C}^{n\times \ell}$, $\mathbf{1}_\ell\in\mathbb{R}^\ell$ is the vector of all ones, and
$$S=\begin{bmatrix}
     s_1 & & \\
     & \ddots & \\
     &&s_\ell\\
    \end{bmatrix}\in\mathbb{C}^{\ell\times\ell},$$
is a diagonal matrix containing the shifts on the main diagonal.

In this setting, an approximate solution to~\eqref{eq:Sylv} of the form $X_m=V_mY_m$ is sought.
In~\cite{Simoncini2010}, the columns of $V_m\in{\color{blue}\mathbb{C}}^{n\times 2m}$ represent an orthonormal basis for the \emph{extended} Krylov subspace
\begin{equation}\label{eq:EK}
\mathbf{EK}_m(A,b)=\text{span}\{b,A^{-1}b,Ab,A^{-2}b,\ldots,A^{m-1}b,A^{-m}b\},
\end{equation}
whereas $Y_m\in\mathbb{C}^{2m\times \ell}$ is computed by imposing a Galerkin condition on the residual matrix $AV_mY_m+V_mY_mS-b\mathbf{1}_\ell^T$. As before, this condition is equivalent to computing $Y_m$ as the solution of a projected problem, which in this case amounts to the following Sylvester equation
$$T_mY_m+Y_mS=\beta e_1\mathbf{1}_\ell^T,\quad T_m=V_m^*AV_m,\;\beta=\|b\|.$$
In~\cite{Simoncini2010} it has been demonstrated that the use of the extended Krylov subspace~\eqref{eq:EK} in place of its polynomial counterpart~\eqref{eq:PK} as approximation space leads to a much faster convergence in terms of number of iterations so that restarting is not really necessary. However, the computation of $V_m$ is now significantly more expensive than before due to solving linear systems with $A$ at each iteration.

Our novel method takes inspiration from what has been proposed in~\cite{Simoncini2010} but differs in both the selection of the approximation space and the computation of the matrix $Y_m$. In our scheme, we assume that we can afford solving a handful of shifted linear systems by, e.g., a sparse direct solver, but that we do not have {\color{blue}enough} resources for solving all the $\ell$ linear systems in~\eqref{eq:main}.
This can be the case when, e.g., we need to solve~\eqref{eq:main} on a laptop where we may be able to solve $(A+s I)x=b$ but the lack of a proper parallel computing environment makes the sequential solution of all the systems in~\eqref{eq:main} extremely time-consuming, {\color{blue} especially for large $\ell$. This issue may be overcome by working with HPC infrastructures. However, the parallel solution of $\ell$ large-scale linear systems is still an extremely energy-draining operation so that being able to solve much fewer linear systems can remarkably lower the carbon footprint of the overall solution procedure}.

In our novel approach, we propose to employ the rational Krylov subspace
\begin{equation}\label{eq:RK}
 \mathbf{K}_{m+1}(A,b,\bm{\xi}_{m})=\text{span}\left\{b,(A+\xi_1I)^{-1}b,\ldots,\prod_{i=1}^{m}(A+\xi_iI)^{-1}b\right\},
\end{equation}
in place of~\eqref{eq:EK} as approximation space.
In~\eqref{eq:RK}, $\bm{\xi}_{m}=(\xi_1,\ldots,\xi_{m})^T\in\mathbb{C}^{m}$ is a set of poles
that can be either given a priori or computed on the fly as the space expands. Various strategies for the computation of these poles are available in the literature; see, e.g.,~\cite{Poles1,Gue13b}. Selecting good poles is crucial to obtain a {\color{blue}performant} rational Krylov subspace method for the problem at hand. In {\color{blue}section~\ref{Poles selection}}, we illustrate an ad-hoc {\color{blue}greedy} routine taking into account the peculiar structure of our problem~\eqref{eq:Sylv}. Numerical experiments show that our approach often performs better than off-the-shelf pole-selection strategies.

{
\color{blue}
The rational Krylov relation that holds at each iteration of the rational Arnoldi procedure \cite{blockrationalKrylov} to construct $V_{m+1}$, the orthonormal basis for $\mathbf{K}_{m+1}(A,b,\bm{\xi}_{m})$, is
\begin{equation}
    \label{eq:rationalArnoldi_HK}
    AV_{m+1}\underline{K}_m = V_{m+1}\underline{H}_m,
\end{equation}
where $\underline{K}_m$ and $\underline{H}_m$ are $(m+1)\times m$ unreduced Hessenberg matrices that result from the orthonormalization procedure.
In this work we propose to seek a solution under the form $X_m = V_{m+1}\underline{K}_mY_m$,
}
where we compute $Y_m\in\mathbb{C}^{m\times \ell}$ by imposing a minimal residual condition instead of a Galerkin one as done in~\cite{Simoncini2010}.
In contrast to what happens when~\eqref{eq:PK} is adopted as approximation space, this appealing minimization property does not come with the drawback of forcing collinearity of the residuals, which may not even exist. Indeed, in our extensive numerical experience, the rational Krylov subspace one needs to construct is always of very moderate dimension, so that no restart has to be performed.

{\color{blue}
In~\cite{GuZhouLin2007} the authors proposed to employ flexible GMRES for~\eqref{eq:main}, preconditioned with operators of the form $\mathcal{P}_j=A+\alpha_jI$ for a small number of predefined scalars $\alpha_j$. Even though this can be viewed as a first, timid step towards our rational Krylov subspace method, the storage of the preconditioned and unpreconditioned bases in flexible strategies makes the need of restarting even more urgent; this carries over all the difficulties related to restarting we mentioned above.

We would like to mention that a peculiar rational Krylov subspace is adopted in~\cite{illposed} for the solution of ill-posed problems. In particular, the authors propose to construct what they call a \emph{mixed} space that we can formally view as~\eqref{eq:RK} where at every other iteration an infinite pole is employed, i.e., $\xi_{2k}=\infty$, $k\geq 1$, whereas the poles with odd index are possibly good regularization parameters chosen from a given set. Even though the ultimate goal in~\cite{illposed} is rather different from ours, we believe that the rational Krylov approach presented here may be seen as a generalization of the mixed approach from~\cite{illposed} when it comes to Tikhonov regularization. The interesting extension of our approach to the solution of ill-posed problems will be studied elsewhere in the near future.}

Here is a synopsis of the paper. In section~\ref{sec: lowrank sol}, we recall some results depicting scenarios {\color{blue}in which} we can indeed expect the solution $X$ to~\eqref{eq:Sylv} to be numerically low rank. This property is key for the success of projection-based solvers for matrix equations. Section~\ref{Rational Krylov approach} {\color{blue}contains} the main contribution of this paper, namely the derivation of a novel minimal residual rational Krylov subspace method for~\eqref{eq:main}. As already mentioned, the performance of this method is related to the quality of the poles selected for the construction of the space~\eqref{eq:RK}. In section~\ref{Poles selection}, an effective pole-selection strategy tailored to the problem~\eqref{eq:main} is presented. The construction of the rational Krylov subspace~\eqref{eq:RK} requires the solution of shifted linear systems with $A$. At first, we derive our new numerical scheme by assuming that we are able to solve these linear systems by a direct method. However, this may not always be the case. Therefore, in section~\ref{A fully iterative variant} we discuss a fully iterative variant of our minimal residual rational Krylov scheme, where the matrix $A$ is required only in the computation of matrix-vector products. 
{\color{blue}In section~\ref{sec:multiple_rhs} we generalize our algorithm to problems~\eqref{eq:main} whose right-hand sides $b^{(i)}$'s are not fixed. This is attained by using block rational Krylov subspaces.} 
The numerical experiments presented in section~\ref{sec:numerical_experiments} show how our novel solver outperforms state-of-the-art techniques for shifted linear systems on very demanding problems like those where, e.g., the shifts $s_i$'s are complex and do not come in complex conjugate pairs. The paper ends with some conclusions in section~\ref{sec:conclusion}.

\section{When to expect a low-rank solution}\label{sec: lowrank sol}
The main contribution of this paper is the design of a projection method for~\eqref{eq:Sylv} that relies on the rational Krylov subspace~\eqref{eq:RK} and is equipped with a minimal residual condition. This scheme belongs to the rather broad family of \emph{low-rank} methods that are effective whenever the exact solution can be well approximated by a low-rank matrix. It is thus natural to ask when the solution $X$ to~\eqref{eq:Sylv} has a fast decay in its singular values, namely when the $\ell$ solutions to~\eqref{eq:main} turn out to be largely linearly dependent.

Intuition suggests that whenever the shifts $s_i$'s are all ``close'' to each other, then the solutions $x^{(i)}$ will be as well, leading to a matrix $X=[x^{(1)},\ldots,x^{(\ell)}]$ of low rank. To the best of our knowledge, this {\color{blue} property has been exploited only in~\cite{Grasedyck04} where Grasedyck shows the decay of the singular values of the solution $X$ to a general Sylvester equation $AX-XB=C$ assuming the eigenvalues of $B$ to be clustered. On the other hand, the techniques emloyed to get these results require $\text{Re}(\lambda-\gamma)<0$ for all eigenvalues $\lambda$ of $A$ and $\gamma$ of $B$. This may be a too restrictive hypothesis in case of~\eqref{eq:Sylv}.}

A different tool to understand when our strategy may be successful is the convergence theory behind rational Krylov methods for Sylvester equations; see, e.g.,~\cite{ConvRKSylv}. However, many of these results rely on the field of values of the coefficient matrices overlooking the diagonal pattern of $S$ in our case.

Assuming $A$ to be diagonalizable with $A=Q\Lambda Q^{-1}$, $\Lambda=\text{diag}(\lambda_1,\ldots,\lambda_n)$, we can transform~\eqref{eq:Sylv} into
\begin{equation}\label{eq:transformedeq}
    \Lambda \widetilde X + \widetilde X S= \widetilde b\mathbf{1}_\ell^T,
\end{equation}
where $\widetilde X=Q^{-1}X$ and $\widetilde b=Q^{-1}b$.
Clearly, any results on the low-rank approximability of $\widetilde X$ carries over to $X$ as $X=Q\widetilde X$. 

Since both $\Lambda$ and $S$ are diagonal, $\widetilde X$ is a Cauchy matrix of the form $(\widetilde X)_{i,j}=\widetilde b_i/(\lambda_i+s_j)$.
Several results on the decay of the singular values of Cauchy matrices can be found in the literature.
In particular, the rate of this decay is often bounded from above by quantities that depend on the distance between the $s_i$'s and the {\color{blue} mirrored} $\lambda_j$'s; see, e.g.,~\cite[Section {\color{blue}4.2}]{BecT17},~\cite[Theorem 3.2]{Truhar1}, and~\cite[Eq. (2.34)]{Truhar2}. None of this contribution takes into account the possible clustering of the shifts $s_i$'s.

In the following proposition, we derive a novel result which relates the low-rank approximability of the {\color{blue}solution $ X$} to the relative distance between the shifts $s_i$'s and a ``centroid'' $\mathbf{s}$, and the distance between the shifts and the eigenvalues of $A$.

\begin{proposition}\label{Prop1}
{\color{blue}Let $A=Q\Lambda Q^{-1}$, $\Lambda=\text{diag}(\lambda_1,\ldots,\lambda_n)$, and} assume that there exists a centroid $\mathbf{s}$ such that $\max_{i=1,\ldots,\ell}|s_i-\mathbf{s}|\leq \epsilon$
    and that $\min_{j=1,\ldots,n}|\lambda_j+\mathbf{s}|/\epsilon\gg 1$. Then
    {\color{blue}
    \begin{equation}
        \| X-(A+\mathbf{s}I)^{-1} b\mathbf{1}_\ell^T\|_F\leq c(\epsilon,\mathbf{s},\Lambda)\kappa_F(Q)\| X\|_F,
    \end{equation}}
    where $c(\epsilon,\mathbf{s},\Lambda)=\epsilon/\min_{j=1,\ldots,n}|\lambda_j+\mathbf{s}|\ll 1$ {\color{blue}and $\kappa_F(Q)=\|Q\|_F\|Q^{-1}\|_F$.}
\end{proposition}
\begin{proof}
{\color{blue} Let's first consider the transformed equation~\eqref{eq:transformedeq}.}
The $(i,j)$-th entry of $\widetilde X$ is of the form $(\widetilde X)_{i,j}=\widetilde b_i/(\lambda_i+s_j)$. We thus have
$$\left|\frac{\widetilde b_i}{\lambda_i+\mathbf{s}}-(\widetilde X)_{i,j}\right|=\left|\frac{\widetilde b_i}{\lambda_i+\mathbf{s}}-\frac{\widetilde b_i}{\lambda_i+s_j}\right|\leq \frac{|\widetilde b_i|\epsilon}{|\lambda_i+\mathbf{s}|\cdot |\lambda_i+s_j|}\leq c(\epsilon,\mathbf{s},\Lambda)|(\widetilde X)_{i,j}|.
$$
By taking the square root of the summation over the indices $i$ and $j$ of the square of the terms above, {\color{blue} it is then easy to show that
\begin{equation}
        \|\widetilde X-(\Lambda+\mathbf{s}I)^{-1}\widetilde b\mathbf{1}_\ell^T\|_F\leq c(\epsilon,\mathbf{s},\Lambda)\|\widetilde X\|_F.
    \end{equation}
We thus have 
$$    \| X-(A+\mathbf{s}I)^{-1} b\mathbf{1}_\ell^T\|_F\leq
        \|Q\|_F\|\widetilde X-(\Lambda+\mathbf{s}I)^{-1}\widetilde b\mathbf{1}_\ell^T\|_F
        \leq
        c(\epsilon,\mathbf{s},\Lambda)\|Q\|_F\|\widetilde X\|_F,$$
     and the result follows by recalling that $\widetilde X=Q^{-1}X$.   
    }
\end{proof}

The assumptions of Proposition~\ref{Prop1} portray a scenario where the shifts $s_i$'s are much closer to the centroid $\mathbf{s}$ than to the mirrored eigenvalues $-\lambda_j$'s, underlying the importance of having well-conditioned problems. Indeed, having close shifts could not be sufficient to obtain (almost) linearly dependent solutions if the related shifted linear systems are badly conditioned.

Proposition~\ref{Prop1} can be easily generalized to the case where we can identify a certain number $m$ of different clusters of shifts, each of them equipped with its centroid $\mathbf{s}_k$, $k=1,\ldots,m$.
In particular, by applying the result of Proposition~\ref{Prop1} column-wise, cluster by cluster, we can show that $ X$ can be well approximated by a rank-$m$ matrix. 

In the following experiment, we numerically validate the findings of Proposition~\ref{Prop1}.

\begin{experiment}\label{Ex:1}
We generate the matrix $A\in\mathbb{R}^{n\times n}$, $n=200$, as $A=Q\Lambda Q^{-1}$, where $Q\in\mathbb{R}^{n\times n}$ is a matrix with random entries drawn from the Gaussian distribution\footnote{{\color{blue}{\tt rng('default')}}}, {\color{blue} $\kappa_F=\mathcal{O}(10^3)$,} whereas $\Lambda=\text{diag}(-100,\ldots,-1,1,\ldots,100)$. We consider $\ell=300$ complex shifts of the form $s_j=0.5+i \omega_j 10^{-k}$, $k>0$, with $\omega_j$ a random scalar drawn once again from the Gaussian distribution. We can think of the exponent $k$ as an index of the level of clustering of the shifts: larger $k$ yields more clustered $s_j$'s. We select $b\in\mathbb{R}^n$ as a random vector with random entries 
drawn from the Gaussian distribution {\color{blue} and unit norm}.

In Table~\ref{tab:my_label} (left), we report the relative error $\|X-(A+\mathbf{s}I)^{-1}b\mathbf{1}_\ell^T\|_F/\|X\|_F$ along with the ratio $\sigma_2(X)/\sigma_1(X)$ by varying $k$, where $X$ denotes the exact solution to~\eqref{eq:Sylv} computed column-wise by the Matlab solver ``backslash''.
\begin{table}[t]
    \centering
    \begin{tabular}{c|cc}
       $k$  & $\frac{\|X-(A+\mathbf{s}I)^{-1}b\mathbf{1}_\ell^T\|_F}{\|X\|_F}$ &$\frac{\sigma_2(X)}{\sigma_1(X)}$ \\
       \hline
    3     &{\color{blue}1.03e-3}&  {\color{blue}7.06e-4}\\
4& {\color{blue}1.09e-4}&   {\color{blue}7.47e-5}\\
5  & {\color{blue}1.03e-5}&  {\color{blue}7.06e-6}\\
6 &    {\color{blue}1.05e-6}&     {\color{blue}7.19e-7}\\
7 &      {\color{blue}1.06e-7} &   {\color{blue}7.23e-8}\\
    \end{tabular}
    \quad
    \begin{tabular}{c|cc}
       $k$  & $\frac{\|X-(A+\mathbf{s}I)^{-1}b\mathbf{1}_\ell^T\|_F}{\|X\|_F}$ &$\frac{\sigma_2(X)}{\sigma_1(X)}$ \\
       \hline
    3     &    {\color{blue}1.10}&     {\color{blue}1.43e-4}
\\
4&     {\color{blue}5.57}&       {\color{blue}7.01e-5}
\\
5  &     {\color{blue}1.01}&     {\color{blue}1.19e-6}
\\
6 &       {\color{blue}1.44}&        {\color{blue}2.68e-7}
\\
7 &         {\color{blue}1.67}
 &     {\color{blue}2.18e-8}
\\
    \end{tabular}
    
    \caption{Example~\ref{Ex:1}. $A=Q\Lambda Q^{-1}\in\mathbb{R}^{n\times n}$, $n=200$, $Q\in\mathbb{R}^{n\times n}$ random matrix with entries drawn from the Gaussian distribution, {\color{blue}$\kappa_F=\mathcal{O}(10^3)$}, $\Lambda=\text{diag}(-100,\ldots,-1,1,\ldots,100)$. $b\in\mathbb{R}^n$ random vector. We consider two different sets of $\ell=300$ shifts.  Left: $s_j=0.5+i \omega_j 10^{-k}$, $j=1,\ldots,300$. Right: $s_j=1+10^{-10}+i \omega_j 10^{-k}$, $j=1,\ldots,300$. $k$ denotes the level of clustering of the constructed shifts. $X\in\mathbb{C}^{n\times \ell}$ denotes the exact solution computed column-wise by ``backslash''.}    
    \label{tab:my_label}
\end{table}

Note that regardless of $k$, the minimum distance between the shifts and the eigenvalues of $A$ is 0.5 and that the shifts lie within the {\color{blue} field of values} of $A$. {\color{blue}Moreover, $\text{Re}(\lambda+s_j)$ does not have a constant sign for all eigenvalues $\lambda$ of $A$ and shifts $s_j$, but it can be either positive or negative}. Therefore, to the best of our knowledge, state-of-the-art results cannot predict how the low-rank approximability of $X$ depends on the clustering of the shifts, as illustrated by the results in Table~\ref{tab:my_label} (left). This aspect is accounted for by Proposition~\ref{Prop1}. Indeed, from the way we constructed the shifts, we have $\mathbf{s}=0.5$ and $\max_{i=1,\ldots,\ell}|s_i-\mathbf{s}|=\mathcal{O}(10^{-k})$ so that $c(\epsilon,\mathbf{s},\Lambda)=\mathcal{O}(10^{1-k})$.

We repeat the same experiment changing the shifts to $s_j=1+10^{-10}+i \omega_j 10^{-k}$ so that $\mathbf{s}=1+10^{-10}$, $\epsilon=\mathcal{O}(10^{-k})$, and $\min_{j=1,\ldots,n}|\lambda_j+\mathbf{s}|=10^{-10}$. 
In Table~\ref{tab:my_label} (right), we report the results. 
For this problem, Proposition~\ref{Prop1} cannot guarantee the low-rank approximability of $X$. Indeed, we have $c(\epsilon,\mathbf{s},\Lambda)=\mathcal{O}(10^{10-k})$ which is large for the values of $k$ we tested. This is numerically confirmed by the large errors reported in Table~\ref{tab:my_label} (right). Nevertheless, $X$ can still be well-approximated by a rank-1 matrix as indicated by the small values of $\sigma_2(X)/\sigma_1(X)$. This suggests that further theoretical investigation is necessary to identify all the scenarios where we can have a numerically low-rank solution. 

\end{experiment}

{\color{blue}
\begin{remark}
    We mention here that we employ the concept of centroids only for theoretical purposes, to depict scenarios where we do expect the singular values of the exact solution to rapidly decay to zero. While centroids might be employed as poles in the rational Krylov basis construction, this strategy is likely not effective in settings 
    where the results from Proposition~\ref{Prop1} do not apply, cf.
Example~\ref{Ex:1}. Therefore, in section~\ref{Poles selection} we propose a general, greedy pole-selection strategy where the poles $\xi_i$'s in~\eqref{eq:RK} are computed on the fly during the iterative process, as it is customary in rational Krylov subspace methods.
\end{remark}

\begin{remark}
    In \cite{KressnerTobler11} the authors report some results on the numerical linear dependency of solutions to parametrized linear systems of the form $A(\alpha)x(\alpha)=b(\alpha)$
    where $\alpha\in[-1,1]$, and in \cite[Section 6.6]{Rec1} recycling Krylov methods have been tested in this scenario.  In particular, in \cite[Theorem 2.4]{KressnerTobler11} it has been shown how the decay of the singular values of the matrix $[x(\alpha_1),\ldots,x(\alpha_p)]$ is bounded from above by $\max_{\alpha\in\mathcal{E}_\rho}\|A(\alpha)^{-1}\|_2$ where $\mathcal{E}_\rho\subset\mathbb{C}$ denotes the open elliptic
disc with foci $\pm1$ and the sum of the half axes equal to $\rho$. This result can be certainly applied to our context as well. By considering the technical assumption that $s_i\in[-1,1]$ for all $i=1,\ldots,\ell$, \cite[Theorem 2.4]{KressnerTobler11} says that the decay of the singular values of the solution $X$ to \eqref{eq:Sylv} is mainly driven by $\max_{s\in\mathcal{E}_\rho}\|(A+sI)^{-1}\|_2\geq \max_{i=1,\ldots,\ell}\|(A+s_iI)^{-1}\|_2$ underlying the importance of having well-conditioned problems in \eqref{eq:main} but overlooking, once again, the possible clustering of the shifts $s_i$'s.

By building upon \cite{KressnerTobler11}, it might be interesting to study how (and if) a possible clustering of the parameters $\alpha$ impacts the decay of the singular values of $[x(\alpha_1),\ldots,x(\alpha_p)]$ thus extending our findings to general parameter-dependent problems $A(\alpha)x(\alpha)=b(\alpha)$. This fascinating topic is left to be studied elsewhere.
\end{remark}
}
%%%%%%%%%%%%%%%%%%%%%%%%%%%%%%%%%%
\section{Minimal residual rational Krylov approach}\label{Rational Krylov approach}
Due to the diagonal pattern of $S$, we employ a projection method with only left projection for the solution of the Sylvester equation~\eqref{eq:Sylv}.  In this framework, the sought after solution is of the form {\color{blue}$X_m=V_{m+1}\underline{K}_mY_m$ with the columns of $V_{m+1}=[v_1,\ldots,v_{m+1}]\in\mathbb{C}^{n\times (m+1)}$} forming an orthonormal basis for a suitable subspace and
$Y_m\in\mathbb{C}^{m\times \ell}$ computed by imposing convenient conditions.
In this paper, we propose the columns of $V_{m+1}$ to span the rational Krylov subspace~\eqref{eq:RK}. The basis $V_{m+1}$ of~\eqref{eq:RK} can be computed via the rational Arnoldi scheme~\cite{Ruhe94} or the rational Lanczos method~\cite{ratLanczos} in the case of Hermitian $A$. Notice that {\color{blue} if} $A$ {\color{blue} is} real, {\color{blue} and} $\xi_j\in\mathbb{R}$, for all $j$, $V_{m+1}$ is real as well. A real basis can also be constructed if $A$ is real and the complex poles come in conjugate pairs; see, e.g.,~\cite{realbasis}.
We compute $Y_m=[y_1,\ldots, y_{\ell}]\in\mathbb{C}^{m\times \ell}$ by imposing a minimal residual condition, namely
{
\color{blue}
\begin{align}\label{eq:minimal_residual_condition}
  Y_m&=\argmin_{Y\in\mathbb{C}^{m\times \ell}}\|AV_{m+1}\underline{K}_mY + V_{m+1}\underline{K}_mYS-b\mathbf{1}_\ell^T\|_F^2\notag\\
  &= \argmin_{Y=[y_1,\ldots, y_{\ell}]}\sum_{j=1}^{\ell}\|AV_{m+1}\underline{K}_my_j + s_jV_{m+1}\underline{K}_my_j-b\|_2^2. 
\end{align}
Using the rational Arnoldi relation~\eqref{eq:rationalArnoldi_HK} we can write the residual matrix as follows
\begin{align}\label{eq:Res_expression}
    R_m=&\,AV_{m+1}\underline{K}_m Y_m + V_{m+1}\underline{K}_mY_mS-b\mathbf{1}_\ell^T\notag\\
    =&\,V_{m+1}\left(\underline{H}_m Y_m + \underline{K}_m Y_mS-\beta e_1\mathbf{1}_\ell^T\right),
\end{align}
so that 
\begin{align}\label{eq:res_norm}
\|R_m\|_F^2=&\left\|\underline{H}_m Y_m + \underline{K}_m Y_mS-\beta e_1\mathbf{1}_\ell^T\right\|_F^2,
    \end{align}
since $V_{m+1}$ is unitary. Therefore, the matrix $Y_m$ in~\eqref{eq:minimal_residual_condition} can be computed by solving the $(m+1)\times m$ least squares problem
\begin{equation}
    \displaystyle{Y_m=\argmin_{Y\in\mathbb{C}^{m\times m}}\left\|\underline{H}_m Y + \underline{K}_m YS-\beta e_1\mathbf{1}_\ell^T\right\|_F.
}\end{equation}

In particular, $Y_m=[y_m^{(1)},\ldots,y_m^{(\ell)}]$ can be computed column-wise by solving $\ell$ small-dimensional least squares problems as follows
\begin{equation}\label{eq:Y_columnwise}
y_m^{(j)}=\argmin_{y\in\mathbb{C}^{m}}\left\|(\underline{H}_m + s_j \underline{K}_m) y-\beta e_1\right\|.
\end{equation}
Using the QR decomposition of the coefficient matrix in the least-squares problem \eqref{eq:Y_columnwise}, we have 
$$
[Q_m^{(j)},P_m^{(j)}]\begin{bmatrix}
G_m^{(j)}\\
0\\
\end{bmatrix}=\underline{H}_m+s_j\underline{K}_m,
$$
where $Q_m^{(j)}\in\mathbb{C}^{(m+1)\times m}$, $P_m^{(j)}\in\mathbb{C}^{(m+1)}$, and $G_m^{(j)}\in\mathbb{C}^{m\times m}$ upper triangular. Therefore,
$$y_m^{(j)}=\beta (G_m^{(j)})^{-1}((Q_m^{(j)})^{\color{blue}*}e_1),$$
and the norm of the residual vector $r_m^{(j)}=(A+s_jI_n)V_my_m^{(j)}-b$ can be cheaply computed as
\begin{equation}\label{eq:resnormcomputation}
   \| r_m^{(j)}\|=    
\|(\underline{H}_m+s_j\underline{K}_m)y_m^{(j)}-\beta e_1\|=\beta\|(P_m^{(j)})^{\color{blue}*}e_1\|.
\end{equation}

In Algorithm~\ref{alg:MRRat}, we report the overall minimal residual rational Krylov subspace scheme for the solution of~\eqref{eq:Sylv}. 
}

{
\color{blue}
\begin{algorithm}[t]
\begin{algorithmic}[1]
%\setstretch{1.2}
\smallskip
\Statex \textbf{Input:} matrix $A\in{\mathbb{C}}^{n\times n}$, shifts $s_i\in\mathbb{C}$, right-hand side $b\in\mathbb{R}^n$, maximum number of iterations $\texttt{maxit}$, tolerance $\texttt{tol}$.
\Statex \textbf{Output:} $V_{m+1}\in\mathbb{C}^{n\times (m+1)}$, $Z_m\in\mathbb{C}^{(m+1)\times \ell}$ such that $\|(A+s_iI_n)V_{m+1}Z_me_i-b\|\leq \|b\| \cdot \texttt{tol}$
\smallskip
\State Set $V_1=v_1=b/\|b\|$
\State Set $\Sigma=\{1,\ldots,\ell\}$, $\Sigma_C=\emptyset$, and randomly select an initial pole $\xi_1$.
\For{$m=1,\ldots,\mathtt{maxit}$}
\State Solve $(A+\xi_{m}I_n)w=v_m$\label{alg:solsystem}
%\State $k_{1:m,m} = V_m^*w$, $k_{m+1,m} v_{m+1} = w - \sum_{i=1}^m v_i k_{i,m}$ 
\label{alg:ortholine} 
\State Set up $\underline{K}_m$ and $\underline{H}_m$ so that $A V_{m+1}\underline{K}_m =V_{m+1} \underline{H}_{m}$
\For{$j\in\Sigma\setminus\Sigma_C$}
\State Compute $y_m^{(j)}$ as in~\eqref{eq:Y_columnwise} and set $z_m^{(j)} = \underline{K}_my_m^{(j)}$
\State Compute the residual norm $\|r_m^{(j)}\|$ as in~\eqref{eq:resnormcomputation}
\If{$\|r_m^{(j)}\|<\|b\|\cdot\mathtt{tol}$}
\State Set $\Sigma_C=\Sigma_C\cup\{j\}$ and $Z_me_j=z_m^{(j)}$
\EndIf
\If{$\Sigma\setminus\Sigma_C = \emptyset$}
\State Go to line~\ref{alg:lastline}
\EndIf
\EndFor
\State Compute next pole $\xi_{m+1}$\label{alg:linecomputepole}
\EndFor
\State Return $V_{m+1}$ and $Z_m$ by possibly padding the latter with zeros\label{alg:lastline}
\end{algorithmic}    \caption{Minimal residual rational Krylov subspace method for~\eqref{eq:Sylv}. \label{alg:MRRat} }

\end{algorithm}
Note that since, $\text{Range}(V_{i+1}\underline{K}_i) \subseteq \text{Range}(V_{j+1}\underline{K}_j)$ for $i\leq j$, the
appealing residual minimization property of our routine guarantees that whenever the residual norm related to the $j$-th linear system in~\eqref{eq:main}, namely $\|r_m^{(j)}\|$, is smaller than the prescribed accuracy target, it remains so also when the subspace is expanded as $\|r_{m+k}^{(j)}\|\leq\|r_{m}^{(j)}\|$ for any $k\geq 0$. 

\begin{remark}\label{remark_Ypadded}
    The monotone decrease of the residual norm can be exploited in computation. In particular, if at iteration $m$ we detect a sufficiently small residual norm for the $j$-th linear system, we stop solving the projected problem related to that linear system in the following iterations. By doing so, the number of small-dimensional least squares problems~\eqref{eq:Y_columnwise} can potentially decrease along the rational Krylov iterations. Note that this strategy requires careful handling of the matrix $Y_m$. In particular, if at the $m$-th iteration $\|r_m^{(j)}\|$ is sufficiently small, then we set $Y_me_j=y_m^{(j)}$. However, it is very unlikely that at this point all the linear systems have converged, namely $\|r_m^{(i)}\|$ is sufficiently small for all $i=1,\ldots,\ell$. We thus proceed with expanding the space but we no longer solve the least squares problem related to the $j$-th linear system. Therefore, for any $k\geq0$ we set $Y_{m+k}e_j=[y_m^{(j)};0]\in\mathbb{C}^{m+k}$ where, with abuse of notation, $0$ denotes a $k$-dimensional vector of all zeros.
\end{remark}}

\begin{remark} {\color{blue}For the sake of simplicity,}
Algorithm~\ref{alg:MRRat} is given by assuming that all the shifts $s_i$ are available. However, working in a framework where new shifts are provided every now and then does not lead to any particular difficulties. Indeed, the construction of the rational Krylov subspace is independent of the shifts $s_i$'s. Whenever a new shift comes at hand, one can simply assemble the related projected least squares problem~\eqref{eq:Y_columnwise} with the rational Krylov quantities already available. {\color{blue} This means that we can fully capitalize on the computational efforts done so far also for the solution of linear systems with newly acquired shifts and the solution process does not need to start from scratch all over again.}  
\end{remark}

\begin{table}[t]
    \centering
    {\color{blue}
    \begin{tabular}{r|l}
         Linear systems solution & $\mathcal{O}(m\cdot\text{nnz}(A))$ \\
         Orthogonalization & $\mathcal{O}(nm^2)$ \\
         Solution of~\eqref{eq:Y_columnwise} & $\mathcal{O}(\ell m^4)$ \\
          
    \end{tabular}
    \caption{Main asymptotic costs (in flops) of performing $m$ iterations of Algorithm~\ref{alg:MRRat}. }
    \label{tab:costs}
    }
\end{table}

{\color{blue}
We conclude this section by reporting in Table~\ref{tab:costs} the main asymptotic costs of Algorithm~\ref{alg:MRRat}, after $m$ iterations. We remind the reader that, at iteration $k$, the cost per iteration of the orthogonalization step and the solution of the projected problems~\eqref{eq:Y_columnwise} amounts to $\mathcal{O}(nk)$ and $\mathcal{O}(\ell k^3)$ flops, respectively, so that after $m$ iterations we end up with the values reported in Table~\ref{tab:costs}. These values highlight the importance of constructing a very meaningful approximation space, able to converge in a few iterations. Otherwise, Algorithm~\ref{alg:MRRat} would not be competitive with the naive solution of the $\ell$ linear systems in~\eqref{eq:main}, whose cost is $\mathcal{O}(\ell\cdot\text{nnz}(A))$ flops. Similarly, the strategy mentioned in Remark~\ref{remark_Ypadded} can significantly help reducing the cost of the solution of the projected problems as not $\ell$ least squares problems need to be solved at each iteration. The cost reported in Table~\ref{tab:costs} is the worst-case scenario, where we have to solve all the $\ell$ projected problems in~\eqref{eq:Y_columnwise} until the $m$-th iteration.
\begin{remark}
    Note that the QR decomposition of $\underline{H}_m+s_j\underline{K}_m$ for each $j=1,\ldots,\ell$, can be updated at each iteration by using a Givens rotations approach similar to how the latter is employed in GMRES~\cite{Schultz1986} and the residual norm computation can be performed almost for free reducing the cost of \eqref{eq:Y_columnwise} to $\mathcal{O}(\ell m^2)$ when the solution $Y_m$ is only computed once the residual is below the required threshold. However, this reduction in cost comes with the extra memory requirement to store additional  $\ell$ matrices of size $(m+1)\times m$, the $R$ factors; $\ell$ vectors of length $m$, the residuals; and $\ell$ sets of Givens rotations each having $m$ rotations. 
\end{remark}

}
%%%%%%%%%%%%%%%%%%%%%%%%%%%%%%%%%%%%%%%%%%%%%%%%%%%%%
\subsection{Pole selection}\label{Poles selection}
As already mentioned, the effectiveness of the rational Krylov subspace~\eqref{eq:RK} is strongly related to the quality of the adopted poles $\bm{\xi}_{m}=(\xi_1,\ldots,\xi_m)^T\in\mathbb{C}^{m}$ for the problem at hand. Available pole-selection strategies often rely on the properties of the projected matrices $V_{m+1}^*AV_{m+1}$ to compute the next pole $\xi_{m+1}$; see, e.g.,~\cite{Poles1,Poles2}.

{\color{blue}
The strategy we propose for the solution of sequences of shifted linear systems does not involve $V_{m+1}^*AV_{m+1}$ solely. It is rather based on the following proposition.
\begin{proposition}\label{Prop:poles}
Let $s$ be a shift selected to be a pole during the construction of $V_{m+1}$, the rational Krylov basis matrix, then $x = (A+sI)^{-1}b = V_{m+1}\underline{K}_m y$, for some $y\in\mathbb{C}^m$.
\end{proposition}
\begin{proof}
Let $s$ be a given shift. The rational Arnoldi relation holding at iteration $m$ in Algorithm~\ref{alg:MRRat} is given as
$$AV_{m+1}\underline{K}_m = V_{m+1}\begin{bmatrix}
    I_m - K_m \Xi_m\\
    -\xi_{m} k_{m+1,m}e_j^T
\end{bmatrix},$$
where $\Xi_m=\text{diag}(\xi_1,\ldots,\xi_{m})$.
We have
\begin{equation}
\label{eq:rational_arnoldi_shift_and_invert_relation}
    V_{m+1}\underline{K}_m = (A+sI)^{-1}V_{m+1}\underline{T}_m,
\end{equation}
where $\underline{T}_m = \begin{bmatrix}
    I_m - K_m (\Xi_m-sI)\\
    -(\xi_{m}-s) k_{m+1,m}e_j^T
\end{bmatrix}$.
Therefore, $(A+sI)^{-1}b = V_{m+1}\underline{K}_m y$ for some $y\in\C^m$ if and only if $e_1 = \underline{T}_my$.
Consider the last selected pole $s=\xi_{m}$. Unless a happy-breakdown occurs, the rank of the left hand side in~\eqref{eq:rational_arnoldi_shift_and_invert_relation} is $m$. Hence, we have $\underline{T}_m = \begin{bmatrix}
    T_m\\0
\end{bmatrix}$ and $T_m = I_m - K_m(\Xi_m-\xi_{m}I)$ is full rank. Therefore,
\begin{align*}
    (A+\xi_mI)^{-1}b &= (A+\xi_mI)^{-1}V_{m}e_1,\\
    &= (A+\xi_mI)^{-1}V_{m}T_m T_m^{-1}e_1,\\
    &= (A+\xi_mI)^{-1}V_{m+1}\underline{T}_m T_m^{-1}e_1,\\
    &= V_{m+1}\underline{K}_{m} T_m^{-1}e_1.    
\end{align*}

For other poles, $\xi_j$, for $j < m$, notice that $\text{Range}(V_{j+1}\underline{K}_j)\subseteq \text{Range}(V_{m+1}\underline{K}_m)$.
Combining the last two arguments proves the proposition.
\end{proof}
}

Proposition~\ref{Prop:poles} implies that if we choose the new pole $\xi_{m+1}$ equal to one of the shifts $s_i$ in~\eqref{eq:main}, then $r_{m+1+k}^{(i)}=0$ for any $k\geq0$. 
By recalling the underlying low-rank assumption on the solution $X$ to~\eqref{eq:Sylv}, we {\color{blue}expect small norms for the residual vectors  corresponding to the solutions that are approximately linearly dependent with solutions that have been included in the search space}.
This argument suggests the following pole selection strategy: {\color{blue} at the beginning of the iterative method, set $\xi_1$ to some randomly selected shift from the set of available shifts. For
$m\geq 1$, in line~\ref{alg:linecomputepole} of Algorithm~\ref{alg:MRRat} perform the following steps:
$$
\widehat j=\argmax_{j\in\Sigma\setminus\Sigma_C}\|r_m^{(j)}\|,\quad\text{and}\quad
\xi_{m+1}=s_{\widehat j}.
$$
We would like to mention that our pole-selection strategy resembles that of some greedy approaches used in reduced basis methods in the context of model order reduction techniques; see, e.g.,~\cite{JelichBaydounVoigtMarburg2021, QuarteroniRozzaManzoni2011, BaydounVoigtJelichMarburg2020}. However, to the best of our knowledge, it has never been proposed before in our setting.
}

\begin{remark}
{\color{blue}It is easy to handle new shifts with the pole-selection strategy above.} In particular, once we solve the projected least squares problems related to these new shifts, we {\color{blue} pick the new pole to be the shift corresponding to the largest residual norm from the previous and new sets of shifts}.
\end{remark}

%%%%%%%%%%%%%%%%%%%%%%%%%%%%%%%%%%
\subsection{A fully iterative variant}\label{A fully iterative variant}
The construction of rational Krylov subspaces requires the solution of shifted linear systems of the form 
\begin{equation}\label{eq:shiftedsystemRK}
(A+\xi_{m+1}I)w=v_m,  
\end{equation}
at each iteration. To this end, sparse direct solvers can be employed~\cite{ScottTuma2023}. However, there are scenarios where an iterative solution may be more suitable. We consider the adoption of a (preconditioned) Krylov method, such as FOM~\cite{Saad2003} or GMRES~\cite{Schultz1986}, for this task, but, in principle, other iterative solvers like multigrid methods can also be employed.

An easy implementation would only see the adoption of the user's favorite Krylov method in line~\ref{alg:solsystem} of Algorithm~\ref{alg:MRRat}. Yet, more sophisticated options would amount, e.g., to the adoption of a space-augmentation approach where we first project~\eqref{eq:shiftedsystemRK} onto the already available space 
$\mathbf{K}_{m+1}(A,b,\bm{\xi}_m)$ so as to capitalize on the computational efforts made for its construction. If this space does not contain a sufficiently accurate approximation to the solution to~\eqref{eq:shiftedsystemRK}, then we can expand this space with some directions belonging to $\mathbf{K}_k(A,v_m)$. Even though this strategy may look appealing, it is unlikely to be successful. Indeed, by keeping an eye on how we selected the pole $\xi_{m+1}$ (cf. section~\ref{Poles selection}), the residual $(A+\xi_{m+1}I)V_my-b$ has large norm {\color{blue} since $\|(A+\xi_{m+1}I)V_my-b\|\ge \|(A+\xi_{m+1}I)V_{m+1}\underline{K}_my-b\|$}. It is thus likely that the same happens for the linear system~\eqref{eq:shiftedsystemRK}.

Due to our pole selection strategy,~\eqref{eq:shiftedsystemRK} inherits all the difficulties related to~\eqref{eq:main}. This means that any recycling approach aimed at utilizing the polynomial Krylov subspace used for, e.g., the previous linear systems $(A+\xi_{k+1}I)w=v_k$, $k=1,\ldots,m-1$, may result in a poor solution scheme.

From our numerical experience, the most straightforward option works the best. In particular, solving~\eqref{eq:shiftedsystemRK} from scratch by a preconditioned Krylov technique turned out to be the most competitive alternative; see section~\ref{sec:numerical_experiments}.

%However, we would like to design a novel scheme that fully capitalizes on the computational efforts made to construct $\mathbf{K}_m(A,b,\bm{\xi}_m)$ also in the solution of 
%%
%\begin{equation}\label{eq:shiftedsystemRK}
%(A+\xi_{m+1}I)w=v_m.    
%\end{equation}
%%
%In particular, we would like to check whether the constructed space $\mathbf{K}_m(A,b,\bm{\xi}_m)$ already contains enough information to solve~\eqref{eq:shiftedsystemRK}. If not, we expand this space with some polyniomial Krylov directions.
%In particular, we may need to construct a basis of the form $U_{m+k}:=[V_m,W_k]$ with orthonormal columns such that $\text{Range}(W_k)=\mathbf{K}_k(A,v_m)$.

%To this end, since at the $m$-th iteration of Algorithm~\ref{alg:MRRat} the basis $V_m$ of $\mathbf{K}_m(A,b,\bm{\xi}_m)$ is available, we can consider an approximate solution $w_{m+k}$ to~\eqref{eq:shiftedsystemRK} of the form $w_{m+k}=V_mz_{m+k}$ and apply a Galerkin condition. In particular, if $r_{m+k}:=v_m-(A+\xi_{m+1}I)w_{m+k}$, we ask this vector to be orthogonal to $\mathbf{K}_m(A,b,\bm{\xi}_m)$. It is easy to show that this condition is equivalent to computing $z_{m+k}$ as the solution to the projected linear system
%$$(T_m+\xi_{m+1}I_m)z_{m+k}=e_m,$$
%where the matrix $T_m=V_m^TAV_m$ is already available from previous iterations.

%Thanks to the rational Arnoldi relation~\eqref{eq:rationalArnoldi} and the Galerkin condition we impose, a direct computation shows that 
%$$\|r_{m+k}\|=$$

%%%%%%%%%%%%%%%%%%%%%%%%%%%%%%%%%%
\section{The case of multiple right-hand sides}\label{sec:multiple_rhs}
For the sake of clarity in the derivation of the method, in the previous sections we considered the right-hand sides $b_i$ in~\eqref{eq:main} to be {\color{blue} the same}. This means that the right-hand side in the matrix equation formulation of the problem, namely equation~\eqref{eq:Sylv}, has rank one. In this section, we show that our procedure does not change significantly by dropping this assumption, given that the right-hand side of~\eqref{eq:Sylv} remains low-rank. We identify a scenario which is rather common in practice. In particular, we assume that we have a handful of different right-hand sides, say $k$, and, for all of them, we want to solve the shifted linear system in~\eqref{eq:main} for all the $\ell$ shifts, with $\ell$ much larger than $k$. We are thus interested in problems of the form 
$$(A+s_iI_n)x^{(i,j)}=b^{(j)},\quad i=1,\ldots,\ell,\;j=1,\ldots,k.$$
The matrix equation formulation of this problem can be written as
\begin{equation}\label{eq:Sylv2}
    AX+X(S\otimes I_k)=B(\mathbf{1}_\ell^T\otimes I_k),
\end{equation}
{\color{blue}
where $X=[x^{(1,1)},\ldots,x^{(1,k)},x^{(2,1)},\ldots,x^{(\ell,k)}]\in\mathbb{C}^{n\times k\ell}$, $B=[b^{(1)},\ldots,b^{(k)}]\in\mathbb{R}^{n\times k}$, and $\otimes$ denotes the Kronecker product. We still look for an approximate solution of the form $X_m=V_{m+1}\underline{K}_mY_m\in\mathbb{C}^{n\times k\ell }$ where now the orthonormal columns of $V_{m+1}=[\mathcal{V}_1,\ldots,\mathcal{V}_{m+1}]\in{\color{blue}\mathbb{C}}^{n\times (m+1)k}$, $\mathcal{V}_i\in{\color{blue}\mathbb{C}}^{n\times k}$ for all $i=1,\ldots,m+1$, are the first $(m+1)k$ columns of the basis matrix that span the \emph{block} rational Krylov subspace $\mathbf{K}_{m+1}(A,B,\bm{\xi}_{m})=\text{Range}([B,(A+\xi_1I)^{-1}B,\ldots,\prod_{i=1}^{m}(A+\xi_iI)^{-1}B])$. For this space, the block counterpart of~\eqref{eq:rationalArnoldi_HK} holds true. In particular, we can write
\begin{equation}\label{eq:blockrationalArnoldi}
    AV_{m+1}\underline{K}_m=V_{m+1}\underline{H}_m, 
\end{equation}
where $\underline{K}_m, \underline{H}_m \in{\mathbb{C}}^{(m+1)k\times mk}$ are block Hessenberg matrices that result from the block orthogonalization step; see, e.g.,~\cite{CasR24} for further details.
As before, we are going to exploit~\eqref{eq:blockrationalArnoldi} for imposing a minimal residual condition to compute the matrix $Y_m=[\mathcal{Y}^{(1)}_m,\ldots,\mathcal{Y}_m^{(\ell)}]\in\mathbb{C}^{mk\times k\ell}$, $\mathcal{Y}_m^{(i)}\in\mathbb{C}^{mk\times k}$ for all $i=1\,\ldots,\ell$. Following similar steps as in section~\ref{Rational Krylov approach}, we have
{
$$R_m=V_{m+1}\left(\underline{H}_mY_m +\underline{K}_mY_m(S\otimes I_k)-E_1\bm{\beta}(\mathbf{1}_\ell^T\otimes I_k)\right),$$}
    where $E_1=e_1\otimes I_k$ and $\bm{\beta}\in\mathbb{R}^{k\times k}$ is such that $B=\mathcal{V}_1\bm{\beta}$. {Since} $V_{m+1}$ {is unitary}, we can compute $Y_m=[\mathcal{Y}_m^{(1)},\ldots,\mathcal{Y}_m^{(\ell)}]$ block-wise by solving 
    $$\mathcal{Y}_m^{(i)}=\argmin_{\mathcal{Y}\in\mathbb{C}^{mk\times k}}\left\|(\underline{H}_m+s_i\underline{K}_m)\mathcal{Y}-E_1\bm{\beta}\right\|_F, \quad i=1,\ldots,\ell. $$

    Once again, this small-dimensional least squares problems can be solved by computing the QR factorization of the coefficient matrix. In particular, if 
$$[Q_m^{(i)},P_m^{(i)}]\begin{bmatrix}
    G_m^{(i)}\\
    0\\
\end{bmatrix}=(\underline{H}_m+s_i\underline{K}_m),$$
    with $Q_m^{(i)}\in\mathbb{C}^{(m+1)k\times mk}$, $P_m^{(i)}\in\mathbb{C}^{(m+1)k\times k}$, and $G_m^{(i)}\in\mathbb{C}^{mk\times mk}$, then 
    $$ \mathcal{Y}_m^{(i)}=(G_m^{(i)})^{-1}((Q_m^{(i)})^*E_1\bm{\beta}),$$
    and
    {
    $$\|\mathcal{R}_m^{(i)}\|_F=\left\|(\underline{H}_m+s_i\underline{K}_m)\mathcal{Y}_m^{(i)}-E_1\bm{\beta}\right\|_F=\|(P_m^{(i)})^*E_1\bm{\beta}\|_F.$$}

    It is easy to show that $R_m=[\mathcal{R}_m^{(1)},\ldots,\mathcal{R}_m^{(\ell)}]$ so that $\|R_m\|_F^2=\sum_{i=1}^\ell\|\mathcal{R}_m^{(i)}\|_F^2$, and we can thus apply the same pole-selection strategy presented in section~\ref{Poles selection}. In particular, we select the next pole $\xi_{m+1}$ for the block rational Krylov basis construction as the shift $s_i$ related to $\max_i\|\mathcal{R}_m^{(i)}\|_F$. Moreover, we can stop solving for $\mathcal{Y}_m^{(i)}$ as soon as we detect a sufficiently small $\|\mathcal{R}_m^{(i)}\|_F$.

    \begin{remark}
The strategy presented above can also be adopted for problems~\eqref{eq:main} where we have $\ell$ right-hand sides $b^{(i)}$ that change with the shift $s_i$. In this case we have to assume that the matrix $B=[b^{(1)},\ldots,b^{(\ell)}]$ admits a low-rank representation of the form $B=B_1B_2^T$ with $B_1\in\mathbb{R}^{n\times k}$, $B_2\in\mathbb{R}^{\ell\times k}$ for small $k$. If this is the case, the construction of the block rational Krylov subspace $\mathbf{K}_{m+1}(A,B_1,\bm{\xi}_{m})$ can be carried out. However, the presence of the matrix $B_2$ must be considered in the computation of the $Y_m$ factor of the approximate solution $X_m=V_{m+1}\underline{K}_mY_m$. In particular, $Y_m=[y_m^{(1)},\ldots,y_m^{(\ell)}]\in\mathbb{C}^{mk\times \ell}$ can be computed column-wise by solving
$$
y_m^{(i)}=\argmin_{y\in\mathbb{C}^{mk}}\left\|(\underline{H}_m+s_i\underline{K}_m)y-E_1\bm{\beta}B_2^Te_i\right\|.
$$
\end{remark}
}

%%%%%%%%%%%%%%%%%%%%%%%%%%%%%%%%%%
\section{Numerical experiments}\label{sec:numerical_experiments}
In this section, we present several numerical experiments illustrating the potential of our minimal residual rational Krylov subspace method (labeled MR-RKSM in the following) in solving long sequences of shifted linear systems.
We compare our solver with two state-of-the-art schemes. 
The first one is the extended Krylov subspace method presented in~\cite{Simoncini2010}\footnote{The implementation of this method has been obtained by slighlty modifying the function {\tt kpik} designed for Lyapunov equations~\cite{Sim07} and available at {\tt https://www.dm.unibo.it/\textasciitilde simoncin/software.html}}. Similarly to our solver, this scheme relies on the matrix equation formulation~\eqref{eq:Sylv} of the problem~\eqref{eq:main} and it solves it by left projection. Compared to our proposed method, there are two main differences: (i) the adopted approximation space is the extended Krylov subspace~\eqref{eq:EK} and (ii) a Galerkin condition is imposed to compute the solution. In the following we refer to this scheme as G-EKSM.
The second one is restarted FOM~\cite{FOM}. We preferred this method over restarted GMRES~\cite{RestartedGMRES} as we will consider difficult problems with nonsymmetric matrices $A$ and possibly complex shifts $s_j$ where restarting GMRES may not be possible.
This is not an issue for FOM where the residuals are always collinear by construction and restarting is always an option. In the following, we will denote restarted FOM as FOM(100) since we will restart it every 100 iterations, if necessary. Moreover, we set the maximum number of restarting cycles equal to 10.

Unless stated otherwise, the stopping criterion of all three methods is based on the magnitude of the relative residual norm $\max_{i=1,\ldots,\ell}\|b-(A+s_iI)x^{(i)}\|/\|b\|$ for which we set the threshold $\mathtt{tol}=10^{-8}$. Moreover, all the (shifted) linear systems involved in the construction of the rational and extended Krylov subspaces are solved by means of the ``backslash'' matlab operator, unless we explicitly report a different solution scheme.

All the experiments reported in this paper have been run using Matlab (version 2024b) on a machine
with a 1.2GHz Intel quad-core i7 processor with 16GB RAM and an Ubuntu 20.04.2
LTS operating system. {\color{blue} The Matlab code for reproducing the experiments that follow is available at {\tt https://github.com/palittaUniBO/MR\_RKSM/tree/main}.}

We would like to mention that during our numerical testing, we examined the performance of the GCRO-DR method proposed in~\cite{Parks2006}; see also~\cite{Parks2016a,Parks2016} for a Matlab implementation and a thorough discussion on certain computational aspects of this algorithm.
This method is designed for the solution of general parametrized linear systems. Even though~\eqref{eq:main} can be seen as an instance of this large class of problems, the GCRO-DR implementation available in~\cite{Parks2016a} is not well-suited for a scenario where all the shifts are given. Indeed, in this implementation one is forced to solve the linear systems in~\eqref{eq:main} one at a time, in a sequential manner. This drawback is overcome in the implementation~\cite{soodhalter_2016_56157}. However, this routine builds a block polynomial Krylov subspace for~\eqref{eq:Sylv} with the whole $b\mathbf 1_\ell^T$ as initial block, ignoring its low rank. This issue led to Out-of-Memory errors in our experiments as we always use large values of $\ell$. For these reasons, no results achieved by GCRO-DR are reported.

\begin{experiment}\label{ex1}
    In the first example, we consider the matrix $A\in\mathbb{R}^{n^d\times n^d}$ coming from the finite difference discretization of the following $d$-dimensional convection-diffusion operator
    $$\mathcal{L}(u)=-\nu\Delta u+w\nabla u,$$
    on $[0,1]^d$ with zero Dirchlet boundary conditions. In the expression above, $\nu>0$ is the viscosity parameter and $w$ is the convection vector. The vector $b\in\mathbb{R}^{n^d}$ in~\eqref{eq:main} has random entries drawn from the normal distribution\footnote{{\color{blue}{\tt rng('default')}}} {\color{blue}and unit norm}.

We start by setting $d=2$, $n=100$, so that $A$ is $10\,000\times 10\,000$, $\nu=0.5$, and $w=(3y(1-x^2),-2x(1-y^2))$. 
We consider three sets of $\ell=1\,000$ shifts, each having an increasing level of difficulty. The first set is composed of only real shifts ({\tt real}) chosen as $\ell$ logarithmically spaced points between $-10^{6}$ and $-10^{-6}$. The second set is given by complex shifts coming in conjugate pairs ({\tt complex - conjugate pairs}). In particular, we pick $\ell/2$ logarithmically spaced points $\theta_j$ between $-10^{6}$ and $-10^{-6}$ and consider the set of $\ell$ shifts given by all the pairs $\rm i\cdot\theta_j$, $-\rm i\cdot\theta_j$ where $\rm i$ denotes the imaginary unit. The construction of the third set where we have complex shifts with no conjugate pairs ({\tt complex - no conjugate pairs}) is a little more involved. In particular, we select $\ell$ shifts such that $s_j=c+\rho(\cos(\theta_j)+v\rm{i} \sin(\theta_j))$, where $\theta_j=2\pi j/\ell$, $c=-223.81+5\rm i$, $\rho=500$, and $v=1$, for $j=1,\ldots,\ell$.

\begin{table}[t!]
    \centering
    \begin{tabular}{r|rrr}
    & \multicolumn{3}{c}{shifts: {\tt real}} \\
    & Its. (Cycles) & Rank($X_m$) & Time (s) \\
    MR-RKSM & {\color{blue}21} (-) & {\color{blue}22} & {\color{blue}0.45}\\
    G-EKSM & 24 (-)& 48 & {\color{blue}0.59}\\
    FOM(100)& 453 (5)& - & {\color{blue}4.10}\\
    \hline
    & \multicolumn{3}{c}{shifts: {\tt complex - conjugate pairs}} \\
    & Its. (Cycles) & Rank($X_m$) & Time (s) \\
    MR-RKSM & 36 (-) & {\color{blue}37} & {\color{blue}1.33}\\
    G-EKSM & 32 (-)& 64 & {\color{blue}1.67}\\
    FOM(100)& 453 (5)& - & {\color{blue}5.58}\\
    \hline
    & \multicolumn{3}{c}{shifts: {\tt complex - no conjugate pairs}} \\
    & Its. (Cycles) & Rank($X_m$) & Time (s) \\
    MR-RKSM & {\color{blue}37} (-)& {\color{blue}38} &  {\color{blue}1.63}\\
    G-EKSM & * & * & *\\
    FOM(100)& * & * &*\\
    \end{tabular}
    \caption{Example~\ref{ex1} (2D). Performance of the different methods by varying the nature of the $\ell=1\,000$ shifts $s_j$ in~\eqref{eq:main}. The reported timings are in seconds. ``*'' means no convergence within the prescribed maximum number of iterations (and restarting cycles).}
    \label{Ex1:tab1}
\end{table}

In Table~\ref{Ex1:tab1}, we collect the results achieved by MR-RKSM, G-EKSM, and FOM(100). In particular, we report the overall number of iterations, the rank of the computed solution, and the running time in seconds. For FOM(100) we also report the number of restarting cycles performed by the method. When we report a number of cycles equal to {\tt cycle} we mean that FOM(100) stopped while performing the {\tt cycle}-th restart. Note that MR-RKSM and G-EKSM do not perform any restart. Moreover, we adopt $\text{Rank}(X_m)$ as an indicator of the memory usage of the different methods. In particular, MR-RKSM and G-EKSM represent the solution in low-rank format by allocating $\text{Rank}(X_m)$ vectors of length $n^2$ and other $\text{Rank}(X_m)$ {\color{blue}vectors} of length $\ell$. Note that FOM(100) stores all the $\ell$ solutions to~\eqref{eq:main} as this is necessary to update the current approximation when restarting is applied. This means that FOM(100) is not able to take advantage of the low rank of $X=[x^{(1)},\ldots,x^{(\ell)}]$ to cut its memory demand.

From the results in Table~\ref{Ex1:tab1} we can see that all three methods perform rather well for the sets of shifts {\tt real} and {\tt complex - conjugate pairs}, with MR-RKSM being the fastest and the one with the lowest memory requirements. It is with the shifts {\tt complex - no conjugate pairs} that the performance of MR-RKSM turns out to be outstanding. Indeed, this is the only method able to converge at all. Moreover, convergence is achieved in a small number of iterations and a rather fast running time. G-EKSM and FOM(100) do not converge in 100 iterations and 10 restarting cycles, respectively. In particular, the relative residual norm attained by G-EKSM exhibits a very oscillating trend in the first 50 iterations and then starts to slowly decrease, settling at $\mathcal{O}(10^{-4})$ at the 100th iteration. We remind the reader that, at this point, the rank of the approximate solution computed by G-EKSM is already 200, much larger than the one achieved by MR-RKSM.
The FOM(100) relative residual norm shows oscillating behavior throughout all iterations and cycles with no significant decrease.
This is clearly visible by looking at Figure~\ref{Ex1:fig1} where we depict the convergence history of the three methods. {\color{blue} Even though we do not have a full understanding of this phenomenon, the superiority of the rational Krylov subspace over its extended and polynomial counterparts when dealing with shifted linear systems with complex shifts was already documented in \cite[Section 3.1]{Poles2} in the context of transfer function approximation. Similarly, it is well-known that polynomial Krylov subspace methods struggle to deal with problems stemming from the discretization of Helmholtz equations~\cite{Ernst2012}; a problem related to the solution of shifted linear systems. In any case, further study on the convergence properties of our scheme is necessary to get a complete picture.}

\begin{figure}[t!]
    \centering
    \includegraphics[scale=0.7]{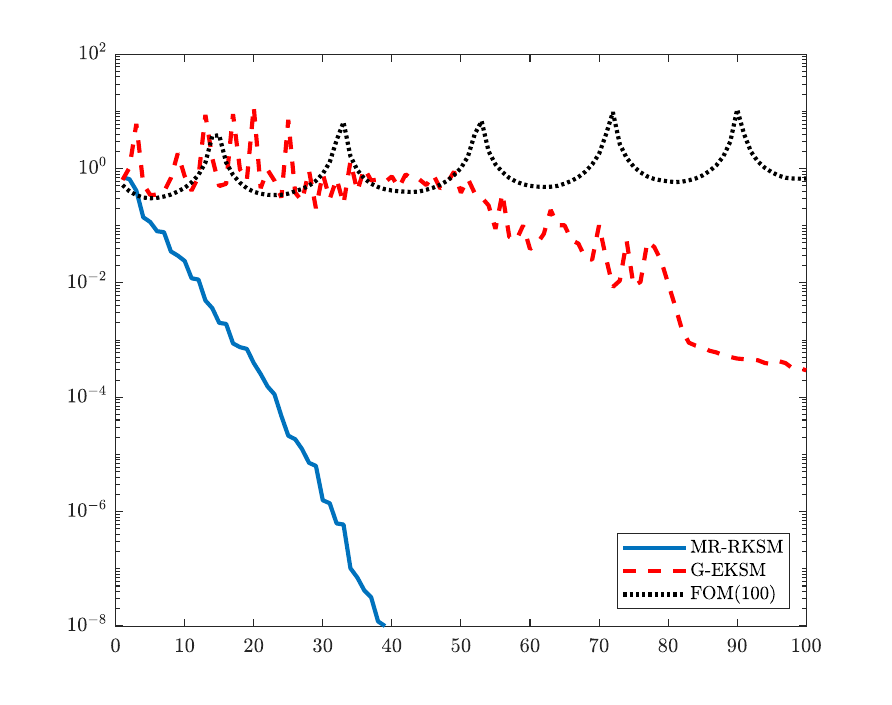}
    \caption{{\color{blue}Example~\ref{ex1} (2D) - shifts: {\tt complex - no conjugate pairs}}. Relative residual norm achieved by MR-RKSM (blue, solid line), G-EKSM (red, dashed line), and FOM(100) (black, dotted line). For FOM(100) we report the relative residual norm of the first 100 iterations, namely during the first cycle. A similar behavior is also observed in the following restarting cycles. }
    \label{Ex1:fig1}
\end{figure}

We now construct a different $A$ by setting $d=3$, $n=50$, so that $A$ is now $125\,000\times 125\,000$, {\color{blue}$\nu=1$}, and $w=(x\cos(x),y\sin(y),\exp(z^2-1))$. The shift sets are as before. It is well-known that iterative solvers work better than their direct counterparts when dealing with linear systems stemming from the discretization of 3D problems. Therefore, we now consider implementations of MR-RKSM and G-EKSM where the linear systems involved in the basis construction are solved by GMRES(50)~\cite{Schultz1986}, where we allow for 100 restarting cycles at most, and preconditioned by an incomplete LU factorization with no fill-in. 
{\color{blue} GMRES(50) is stopped as soon as the computed relative residual norm drops below $10^{-9}$, namely one order of magnitude smaller than the outer tolerance.} 
{\color{blue} Notice, however, 
 that this inner threshold may be relaxed, i.e., increased, as the MR-RKSM iterations proceed by adapting the inexact rational Krylov approach presented in \cite{inexactRational} for the solution of Lyapunov equations. While this approach would make the results in Proposition~\ref{Prop:poles} not holding anymore, the computed rational Krylov space may still be a good approximation space.

We report the results for a fixed inner tolerance in Table~\ref{Ex1:tab2}.}

\begin{table}[t!]
    \centering
    \begin{tabular}{r|rrr}
    & \multicolumn{3}{c}{shifts: {\tt real}} \\
    & Its. (Cycles) & Rank($X_m$) & Time (s) \\
    MR-RKSM & {\color{blue}19} (-) & 20 & {\color{blue}4.62}\\
    G-EKSM & 20 (-)& 40 & {\color{blue}15.74}\\
    FOM(100)& {\color{blue}230} (3)& - & {\color{blue}8.93}\\
    \hline
    & \multicolumn{3}{c}{shifts: {\tt complex - conjugate pairs}} \\
    & Its. (Cycles) & Rank($X_m$) & Time (s) \\
    MR-RKSM & {\color{blue}30} (-) & {\color{blue}31} & {\color{blue}18.83}\\
    G-EKSM & 27 (-)& 54 & {\color{blue}23.55}\\
    FOM(100)& {\color{blue}230} (3)& - & {\color{blue}11.20}\\
    \hline
    & \multicolumn{3}{c}{shifts: {\tt complex - no conjugate pairs}} \\
    & Its. (Cycles) & Rank($X_m$) & Time (s) \\
    MR-RKSM & {\color{blue}34} (-)& {\color{blue}35} & {\color{blue}527.85}\\
    G-EKSM & * & * & *\\
    FOM(100)& * & * &*\\
    \end{tabular}
    \caption{Example~\ref{ex1} (3D). Performance of the different methods by varying the nature of the $\ell=1\,000$ shifts $s_j$ in~\eqref{eq:main}. The reported timings are in seconds. ``*'' means no convergence within the prescribed maximum number of iterations (and restarting cycles).}
    \label{Ex1:tab2}
\end{table}

The iterative solution of the inner (shifted) linear systems makes the computational gap between the matrix-equation-based solvers (MR-RKSM and G-EKSM) and FOM(100) much narrower, especially for {\tt real} and {\tt complex - conjugate pairs} shifts, to the point that FOM(100) turns out to be the fastest method in the case of {\tt complex - conjugate pairs} shifts. Nevertheless, MR-RKSM and G-EKSM are able to capitalize on the low-rank structure of the solution and remarkably reduce the storage demand.
As for the 2D example, also in this case MR-RKSM is the only method able to converge when adopting {\tt complex - no conjugate pairs} shifts. The quite high computational time needed by MR-RKSM in this case is due to GMRES(50). Indeed, for some poles $\xi_j$, GMRES(50) needs a large number of restarting cycles for solving the related shifted linear systems with $A+\xi_j I$. This can be possibly fixed by employing more performing preconditioning operators.

\end{experiment}

\begin{experiment}\label{Ex2}
In the second experiment, we consider as $A$ the matrix {\tt qc2534}\footnote{Available at {\color{blue}{\tt https://math.nist.gov/MatrixMarket/data/NEP/h2plus/h2plus.html}}} which is part of the data set of matrices used for testing eigensolvers in the SLEPc package~\cite{SLEPc}. This matrix is of moderate dimension, $A\in{\color{blue}\mathbb{C}}^{n\times n}$, $n=2\,534$, $b\in\mathbb{R}^n$ is again a {\color{blue}Gaussian vector\footnote{{\color{blue}{\tt rng('default')}}}with unit norm}. We consider $\ell$ shifts of the type {\tt complex - no conjugate pairs} constructed as before. In particular, $s_j=c+\rho(\cos(\theta_j)+v\rm{i} \sin(\theta_j))$ where $\theta_j=2\pi j/\ell$, $c=-0.8-0.07\rm i$, $\rho=0.2$, and $v=0.1$, for $j=1,\ldots,\ell$ and different values of $\ell$. 

Also for this example, G-EKSM and FOM(100) do not perform well due to the nature of the shifts. We thus decided to skip them here. In addition to testing MR-RKSM, we document the running time devoted to sequentially solving all the shifted linear systems in~\eqref{eq:main} by ``backslash''. Indeed, due to the small value of $n$, one may think that the latter is a more suitable option. Nevertheless, we are going to show how our low-rank approach is competitive even in this scenario. Moreover, we compare the pole selection strategy we proposed in section~\ref{Poles selection} with state-of-the-art procedures agnostic to the underlying shifted linear systems problem. In particular, we employ the strategy called ADM in~\cite{Poles1}. {\color{blue} To this end, we slightly modify the function {\tt rk\_adaptive\_sylvester.m},
available in the package {\tt rk\_adaptive\_sylvester}\footnote{Available at {\tt https://github.com/numpi/rk\_adaptive\_sylvester/tree/main}}, designed for general Sylvester equations of the form $AX-XB-uv^T=0$, in order to perform only a left projection and solve~\eqref{eq:Y_columnwise} at the projected level.} 
Notice, however, that, similarly to any other state-of-the-art scheme for pole selection in rational Krylov methods for Sylvester equations, the strategies presented in~\cite{Poles1} are based on an expression of the residual matrix obtained by imposing a Galerkin condition rather than a minimal residual one as we do in our approach. Therefore, the findings from~\cite{Poles1} may need some adjustments to fully adapt to our framework. This is clearly beyond the scope of this paper. Our goal here is just to illustrate how the simple and somehow natural pole selection strategy from section~\ref{Poles selection} works better than off-the-shelf, involved alternatives that are available in the literature.

\begin{table}[t!]
    \centering
    \begin{tabular}{r|rrr}
    & \multicolumn{3}{c}{$\ell=256$} \\
    & Its.  & Rank($X_m$) & Time (s) \\
    MR-RKSM (Poles as in sec.~\ref{Poles selection}) &{\color{blue}38} & {\color{blue}39} & {\color{blue}0.78}\\
    MR-RKSM (ADM poles~\cite{Poles1}) & {\color{blue}127} &{\color{blue}128} &{\color{blue}7.84} \\
    backslash & - & - & {\color{blue}2.92}\\
    \hline
& \multicolumn{3}{c}{$\ell=512$} \\
    & Its.  & Rank($X_m$) & Time (s) \\
    MR-RKSM (Poles as in sec.~\ref{Poles selection}) &{\color{blue}38} & {\color{blue}39} & {\color{blue}1.05}\\
    MR-RKSM (ADM poles~\cite{Poles1}) & {\color{blue}127} & {\color{blue}128}&{\color{blue}11.80} \\
    backslash & - & - & {\color{blue}6.40}\\
    \hline
& \multicolumn{3}{c}{$\ell=1024$} \\
    & Its.  & Rank($X_m$) & Time (s) \\
    MR-RKSM (Poles as in sec.~\ref{Poles selection}) &{\color{blue}38}& {\color{blue}39} & {\color{blue}1.60}\\
    MR-RKSM (ADM poles~\cite{Poles1}) & {\color{blue}127} & {\color{blue}128}& {\color{blue}20.27}\\
    backslash & - & - & {\color{blue}12.57}\\
    
    \end{tabular}
    \caption{Example~\ref{Ex2}. Performance of the different methods. The reported timings are in seconds. }
    \label{Ex2:tab1}
\end{table}

In Table~\ref{Ex2:tab1}, we report the results obtained by changing the number $\ell$ of shifts.
The reported timings showcase the potential of our rational Krylov approach even when dealing with medium-sized matrices $A$ and a moderate number of shifts $\ell$. Moreover, the timings of the solution process based on backslash clearly grow linearly with $\ell$ whereas MR-RKSM (with our pole selection) is not really affected by this drawback. The poles computed by the ADM strategy from~\cite{Poles1} are not as effective as the ones computed by following section~\ref{Poles selection}, {\color{blue} to the point that MR-RKSM with ADM poles is less performing than the naive backslash solution.} 
{\color{blue} In Figure~\ref{Ex2:fig1}, for $\ell=1024$, we report $\|r_m^{(j)}\|$ in logarithmic scale for all the iterations $m$ ($x$-axis) and shift index $j$ ($y$-axis) in case of our pole selection strategy (left) and with the ADM poles (right). We can notice a very different behavior. In particular, in Figure~\ref{Ex2:fig1} (left) we can spot several blue horizontal lines. Those are the residual norms related to those shifted linear systems whose shift is used as rational Krylov pole. Indeed, as shown in Proposition~\ref{Prop:poles}, those residual norms become zero as soon as this happens. This is not the case for the ADM poles that bring the residual norms down to the target accuracy much more slowly.
}
It is also worth noting that the overall running time of MR-RKSM shows a much more important correlation with $\ell$ when ADM is employed, in contrast to what happens when the strategy from section~\ref{Poles selection} is adopted. This is caused by the cost of ADM which now depends on $\ell$ due to the lack of a right projection in our framework.

  \begin{figure}[t]
  \centering
    \hspace{-6cm}
  \begin{minipage}{1\textwidth}
    \hspace{-6.5cm}
  \centering
  \includegraphics[scale=.45]{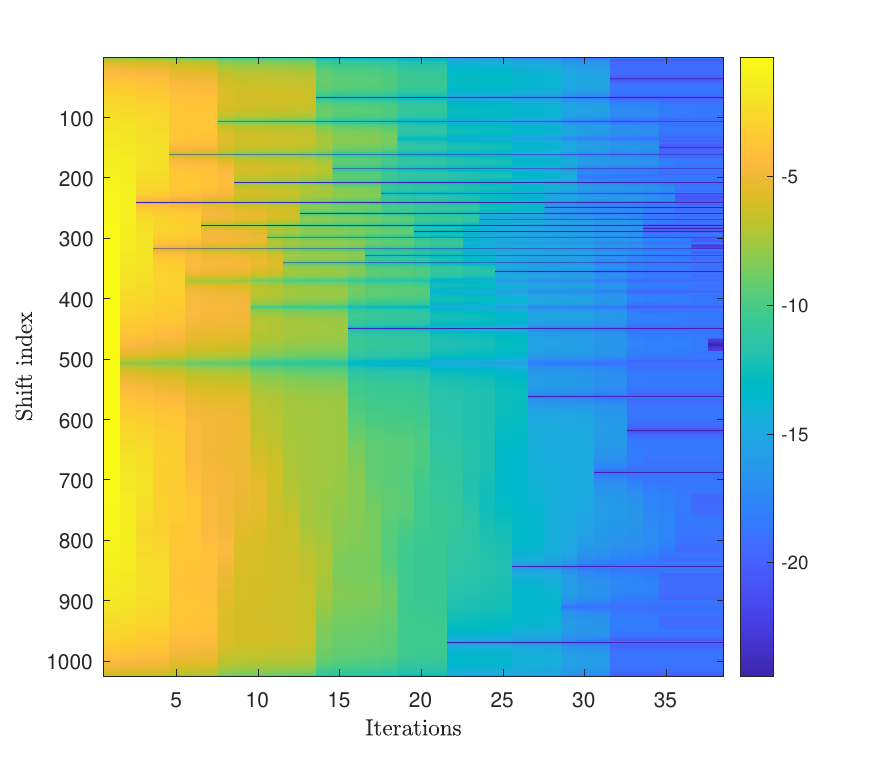}
  \end{minipage}~  \hspace{-6.5cm}
\begin{minipage}{1\textwidth}
    \hspace{-6.5cm}
    \centering
  \includegraphics[scale=.45]{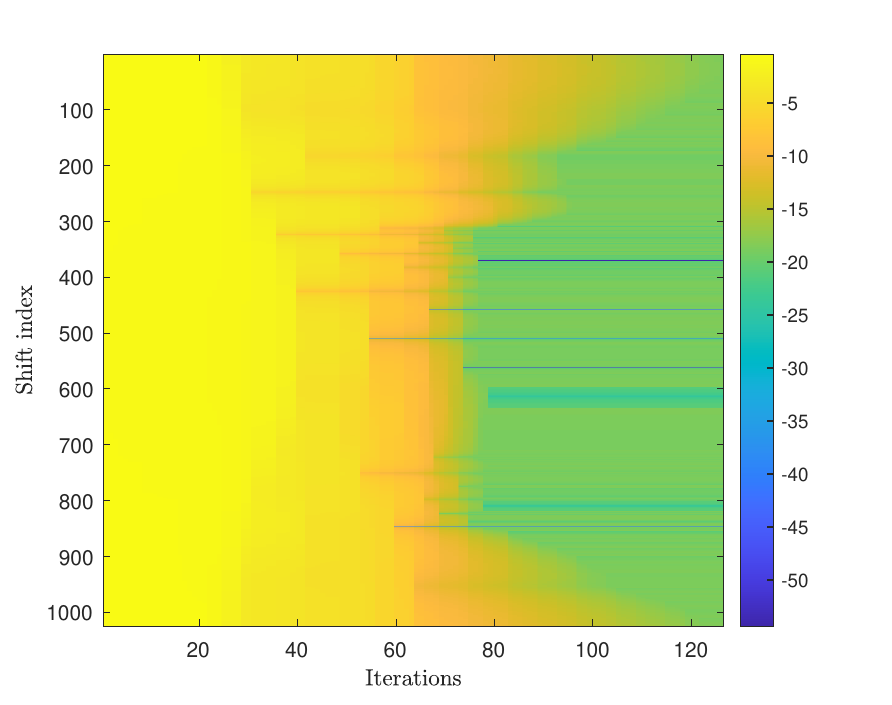}
\end{minipage}

   \caption{{\color{blue}Example~\ref{Ex2}, $\ell=1024$. Heatmap where we report $\|r_m^{(j)}\|$ in logarithmic scale for all iterations $m$ ($x$-axis) and shift indeces $j$ ($y$-axis). Left: MR-RKSM with poles as in sec.~\ref{Poles selection}. Right: MR-RKSM with ADM poles~\cite{Poles1}.
   }
\label{Ex2:fig1}}
  \end{figure}

\end{experiment}
%%%%%%%%%%%%%%%%%%%%%%%%%%%%%%%%%%
\section{Conclusion}\label{sec:conclusion}
We showed once again that a matrix equation formulation is the way to go when dealing with sequences of shifted linear systems of the form~\eqref{eq:main}. This point of view was first explored in~\cite{Simoncini2010} where an extended Krylov subspace method equipped with a Galerkin condition was proposed. This paper makes an important step further. In particular, thanks to our minimal residual rational Krylov subspace method, we are now able to efficiently solve difficult problems, with nonsymmetric coefficient matrices $A$ and complex shifts not coming in conjugate pairs. Like any projection technique for matrix equation, our solver works well if and only if the exact solution presents a fast decay in its singular values. We portrayed scenarios where we expect this to happen. Moreover, we generalized the approach to the block case, where we have a number of different right-hand sides $b^{(j)}$, and proposed a cheap and effective pole selection strategy for the construction of the rational Krylov subspace.

We would like to mention that the minimal residual condition adopted in this paper can be easily integrated in an extended Krylov subspace method by following, e.g., what has been done in~\cite{Benneretal2023} in the context of shifted linear systems within the ADI solver for Lyapunov equations.
Moreover, even though in our numerical experiments the constructed subspace has always been of very moderate dimension, one could try to adapt the compress-and-restart paradigm presented in~\cite{Kressneretal2021} for polynomial Krylov subspaces to rational ones. This {\color{blue} requires an} understanding of {\color{blue}the} collinearity {\color{blue}properties} of residuals in matrix equations, {\color{blue}which}, to the best of our knowledge, has never been explored in the literature.

{
\color{blue}
We also observed that each of the different Krylov approaches may be preferable for a particular class of shifted linear systems; however, identifying the most efficient method is not straightforward in the absence of spectral information about the matrices involved. Although polynomial Krylov methods do not require the solution of linear systems, they often exhibit slow convergence. Moreover, the large number of iterations typically needed for convergence limits the extent to which the low-rank structure of the solution can be effectively exploited.

Extended Krylov methods, on the other hand, involve solving a sequence of linear systems that share the same coefficient matrix, which can be advantageous. Nevertheless, we observed that for certain examples motivated by e.g., the contour integral method for eigenvalue problem, the convergence of these methods may deteriorate, necessitating a large number of iterations and thereby diminishing the benefits of having a low-rank solution.

Our proposed approach which is based on rational Krylov methods requires the solution of a number of shifted linear systems, which is feasible when an efficient linear solver is available. The principal advantage of this method is its rapid convergence combined with very low memory requirements. Indeed, our numerical experiments indicate that the memory usage is close to optimal.
}
%\appendix

\section*{Acknowledgments}
 {\color{blue} We thank the anonymous reviewers for their insightful comments and remarks that helped improving the quality of this paper. In particular, the minimal residual formulation~\eqref{eq:minimal_residual_condition} is a simpler and cleaner form of the originally proposed one. We also thank Lars Grasedyck for pointing us to the results in~\cite{Grasedyck04}.}
 
The first author is supported by the Ada Lovelace Centre Programme at the Scientific Computing Department, STFC.
The second author is member of the INdAM Research
Group GNCS. His work
was partially supported by the European Union - NextGenerationEU under the National Recovery and Resilience Plan (PNRR) - Mission 4 Education and research
- Component 2 From research to business - Investment 1.1 Notice Prin 2022 - DD N. 104 of 2/2/2022,
entitled “Low-rank Structures and Numerical Methods in Matrix and Tensor Computations and their
Application”, code 20227PCCKZ – CUP J53D23003620006.
\bibliographystyle{siamplain}
\bibliography{refs}

\end{document}